\documentclass[12pt,reqno]{amsart}

 \usepackage{graphicx}
 \usepackage{color}

 \usepackage{a4wide}
 \usepackage{amssymb}
 \usepackage{amsmath}
 \usepackage{amsthm}
 \usepackage{amsfonts}
 \usepackage{amscd}
 \usepackage{boxedminipage}
 \usepackage{nomencl}
 \usepackage{verbatim}
 \usepackage{booktabs}
 \usepackage{lscape}
 \usepackage[figuresright]{rotating}

  \usepackage[T1]{fontenc}

 \usepackage{pstricks,pst-node,pst-tree,pst-poly}
 \psset{unit=\unitlength}
 \psset{griddots=6,gridlabels=0,gridcolor=gray,subgriddiv=1}
 \newpsobject{showgrid}{psgrid}{subgriddiv=2,griddots=10,gridlabels=6pt,gridcolor=lightgray}
 \newif\ifhrule\hrulefalse

 \makeatletter
 \CheckCommand*\refstepcounter[1]{\stepcounter{#1}%
  \protected@edef\@currentlabel
  {\csname p@#1\endcsname\csname the#1\endcsname}%
 }
 \renewcommand*\refstepcounter[1]{\stepcounter{#1}%
 \protected@edef\@currentlabel
  {\csname p@#1\expandafter\endcsname\csname the#1\endcsname}%
 }
 \def\labelformat#1{\expandafter\def\csname p@#1\endcsname##1}
 \DeclareRobustCommand\Ref[1]{\protected@edef\@tempa{\ref{#1}}%
  \expandafter\MakeUppercase\@tempa
 }
 \makeatother

 \makeatletter
 \newcommand{\numberlike}[2]{%
  \expandafter\def\csname c@#1\endcsname{%
   \expandafter\csname c@#2\endcsname}%
 }
 \makeatother

 \def\DefaultNumberTheoremWithin{section}

 \theoremstyle{plain}
 \newtheorem{Lemma}{Lemma}
  \numberwithin{Lemma}{\DefaultNumberTheoremWithin}
  \labelformat{Lemma}{Lemma~#1}
 
  \numberwithin{Claim}{\DefaultNumberTheoremWithin}
  \numberlike{Claim}{Lemma}
  \labelformat{Claim}{Claim~#1}
 \newtheorem{Theorem}{Theorem}
  \numberwithin{Theorem}{\DefaultNumberTheoremWithin}
  \numberlike{Theorem}{Lemma}
  \labelformat{Theorem}{Theorem~#1}
 \newtheorem{Corollary}{Corollary}
  \numberwithin{Corollary}{\DefaultNumberTheoremWithin}
  \numberlike{Corollary}{Lemma}
  \labelformat{Corollary}{Corollary~#1}
 \newtheorem{Proposition}{Proposition}
  \numberwithin{Proposition}{\DefaultNumberTheoremWithin}
   \numberlike{Proposition}{Lemma}
   \labelformat{Proposition}{Proposition~#1}
 
   \numberwithin{Conjecture}{\DefaultNumberTheoremWithin}
   \numberlike{Conjecture}{Lemma}
   \labelformat{Conjecture}{Conjecture~#1}

 \theoremstyle{definition}
 
   \numberwithin{Definition}{\DefaultNumberTheoremWithin}
   \numberlike{Definition}{Lemma}
   \labelformat{Definition}{Definition~#1}

 \theoremstyle{definition}
 
   \numberwithin{Question}{\DefaultNumberTheoremWithin}
   \numberlike{Question}{Lemma}
   \labelformat{Question}{Question~#1}

 \theoremstyle{definition}
 
   \numberwithin{Problem}{\DefaultNumberTheoremWithin}
   \numberlike{Problem}{Lemma}
   \labelformat{Problem}{Problem~#1}

 \theoremstyle{remark}
 \newtheorem{Remark}{Remark}
   \numberwithin{Remark}{\DefaultNumberTheoremWithin}
   \numberlike{Remark}{Lemma}
   \labelformat{Remark}{Remark~#1}
 \theoremstyle{remark}

   \numberwithin{Example}{\DefaultNumberTheoremWithin}
   \numberlike{Example}{Lemma}
   \labelformat{Example}{Example~#1}
 
   \labelformat{Case}{Case~#1}
   \numberwithin{Case}{Lemma}
 
   \labelformat{Step}{Step~#1}
   \numberwithin{Step}{Lemma}

 \labelformat{equation}{(#1)}
 \labelformat{figure}{Figure~#1}
 \labelformat{section}{\S#1}
 \labelformat{subsection}{\S#1}
 \labelformat{subsubsection}{\S#1}

 \def\eqref{\ref}

 \allowdisplaybreaks

 \newcommand{\mb}{\mathbb}

\newcommand{\cM}{\mathcal{M}}

 \newcommand{\KK}{\mb{K}}

 \newcommand{\iin}{\mathrm{in}}

 \def\a{\alpha}
 \def\b{\beta}
 
 \def\I{\mathcal{I}}

 \newcommand{\reg}{\mathrm{reg}}

\begin{document}
\title{A Gorenstein simplicial complex for symmetric minors}


 \author{Aldo Conca}
   \address{University of Genova \\
       DIMA - Dipartimento di Matematica \\
       Via Dodecaneso, 35 \\
       16146 Genova\\
       Italy}
   \email{conca@dima.unige.it}

 \author{Emanuela de Negri}
   \address{University of Genova \\
       DIMA - Dipartimento di Matematica \\
       Via Dodecaneso, 35 \\
       16146 Genova\\
       Italy}

   \email{denegri@dima.unige.it}

 \author{Volkmar Welker}
   \address{Fachbereich Mathematik und Informatik\\
       Philipps-Universit\"at \\
       35032 Marburg\\
       Germany}
   \email{welker@mathematik.uni-marburg.de}

 \begin{abstract}
   We show that the ideal generated by the $(n-2)$ minors of a general 
   symmetric $n$ by $n$ matrix has an initial ideal that is the Stanley-Reisner
   ideal of the boundary complex of a simplicial polytope and has the
   same Betti numbers. 
 \end{abstract}

 \thanks{The third author was partially supported by MSRI}

 \maketitle

 \section{Introduction}
 Let $I$ be a homogeneous ideal in a polynomial ring $T=\KK[x_1,\ldots,x_n]$ over a field $\KK$. 
 Assume that $I$ is  Gorenstein,  i.e. the quotient ring $T/I$  is Gorenstein. 
 The general question whether $I$ has a (possibly square free) Gorenstein initial ideal  
 has been discussed recently by several authors for classical families of ideals, see \cite{Athanasiadis,BrunsRomer,ConcaHostenThomas2006,JonssonWelker2007,
 PylyavskyyPetersenSpeyer2010,ReinerWelker,SantosStumpWelker2014,SollWelker2009}.

 In particular, a positive answer is given for important classes of classical ideals: ideals of minors \cite{SollWelker2009}, ideals of Pfaffians \cite{JonssonWelker2007}
 and Pl\"ucker relations \cite{ReinerWelker, PylyavskyyPetersenSpeyer2010,SantosStumpWelker2014}.
 In these examples the associated Gorenstein initial ideals are actually square-free and the corresponding simpicial complexes have  a beautiful combinatorial descriptions. 
 Indeed,  they provide a link to the theory of (multi)-associahedra and to generalized cluster complexes (see \cite{CellabosLabbeStump}).

  Let $S=\KK [x_{ij} ~|~ 1 \leq i \leq j \leq n]$ be   the polynomial ring in the variables $x_{ij}$ over a field $\KK$.
  For $1 \leq j < i \leq n$ we set $x_{ji}=x_{ij}$ and consider
  the generic $n\times n$ symmetric matrix $X = (x_{ij})_{1 \leq i,j \leq n}$.
  For $2 \leq t \leq n$ we denote by $I_t$  the ideal generated by the $t$-minors of $X$.
  It is known that $S/I_t$ is a Cohen-Macaulay normal domain and that it is Gorenstein if and only if $n-t$ is even, see \cite{Goto,Kutz}.

  It is proved in \cite{Conca1994} and in \cite{SturmfelsSullivant} that the $t$-minors of $X$ are a Gr\"obner bases with respect
  to the lexicographic order induced by $$x_{11}>x_{12}>\dots>x_{1n}>x_{22}>\dots>x_{nn}.$$ The corresponding initial ideal is square-free and
  Cohen-Macaulay and it is even the Stanley-Reisner ideal of a shellable simplicial complex. 
  But, apart from few exceptions ($t=n$),   these initial ideals are not Gorenstein because they do not have the right number of ``cone points", see  \cite{ConcaHostenThomas2006} for details. 
  
  On the other hand, a Gorenstein square free initial ideal of $I_2$ has been given in 
  \cite{ConcaHostenThomas2006} and (implicitly) in \cite{BrunsRomer}.
  
  The goal of this paper is to treat the case of the ideal $I_{n-2}$  of  minors of size $n-2$ of a symmetric matrix of variables of size $n\times n$. 
  In the remaining cases,  i.e. $2<t<n-2$ and $n-t$ even, we have not been  able to identify  a Gorenstein initial ideal for $I_t$. Indeed, we do not even have a guess. 
 
Returning to the case $t=n-2$,  we will actually prove that the minors of size $n-2$ of $X$  form a Gr\"obner basis of the ideal $I_{n-2}$ with respect to a suitable reverse  lexicographic order $\prec$ and that the corresponding initial ideal $\iin_{\prec}(I_{n-2})$ is square-free and Gorenstein.
  Furthermore, we will show in Section \ref{r-covered} that the simplicial complex associated to $\iin_{\prec}(I_{n-2})$ is the boundary  complex of a cyclic polytope.

  Typically  the Betti numbers of initial ideals are bigger than the Betti numbers of  the original ideal, and this happens also for determinantal ideals whose initial ideal is Gorenstein, see for instance the examples in \cite{ConcaHostenThomas2006}. 
  This behaviour can be explained  theoretically, at least for minors of order $2$, using the logarithmic bounds for the regularity of    a quadratic monomial ideals with first linear syzygies established in \cite{DaoHunekeSchweig}.

  In our case however,  it turns out that the Betti numbers of $I_{n-2}$ and $\iin_{\prec}(I_{n-2})$ actually coincide. The   reason is that the Betti numbers of a compressed graded Gorenstein $K$-algebra of even regularity just depend on the
  regularity and the codimension, see \ref{compressed}.  By observing that $S/I_{n-2}$ and $S/\iin_{\prec}(I_{n-2})$ are
  Gorenstein compressed of regularity $2(n-3)$ and codimension $6$ one concludes that $\beta_{ij}(S/I_{n-2})=\beta_{ij}(S/\iin_{\prec}(I_{n-2}))$.
  To sum up, the goal of this note is to prove the following

  \begin{Theorem}
    \label{main}
    Let $X$ be the generic  $n\times n$ symmetric matrix. Then there exists a reverse lexicographic term order $\prec$ such that:
    \begin{enumerate}
    \item  the $(n-2)$-minors of $X$ form a Gr\"obner basis of $I_{n-2}$;
    \item  $\iin_\prec(I_{n-2})$  is square-free and defines a Gorenstein ring;
    \item  $\beta_{ij}(I_{n-2})=\beta_{ij}(\iin_\prec(I_{n-2}))$ for all $i,j$;
    \item  $\iin_\prec(I_{n-2})$ is an iterated cone over the Stanley-Reisner ideal of the boundary complex of a $(2n-6)$-dimensional cyclic 
    polytope on $2n$ vertices.
    \end{enumerate}
  \end{Theorem}

 \section{Generalities}
   \label{general}
   In the following theorem we collect important results about the determinantal ring $S/I_t$ of a generic symmetric matrix $X$ of size
   $n\times n$  proved by Kutz \cite{Kutz}, Goto \cite{Goto},  Harris and Tu \cite{HarrisTu} and Conca \cite{Conca1994,Conca1994-1,Conca1994-2}.

   \begin{Theorem}
    \label{pr:basic}
    The ring $S/I_t$ is a Cohen-Macaulay normal domain. It is Gorenstein if and only if $n-t$ is even.
    Krull dimension, multiplicity, Castelnuovo-Mumford regularity and $a$-invariant of $S/I_t$ are given by the following formulas:
    $$\begin{array}{ll}
        \dim(S/I_t)=\frac{(2n+2-t)(t-1)}{2} &
         e(S/I_t)=\displaystyle{\prod_{a=0}^{n-t} }\frac{ \binom{n+a}{t+2a-1}}{ \binom{2a+1}{a}} \\ \\
         \reg(S/I_t)=\left\{
         \begin{array}{rl}
           \frac{(n+2-t)(t-1)}{2} & \mbox{ if $n-t$ is even}\\ \\
           \frac{(n+1-t)(t-1)}{2} & \mbox{ if $n-t$ is odd}
        \end{array}
        \right. &
        a(S/I_t)=\left\{
        \begin{array}{ll}
          -\frac{n(t-1)}{2} & \mbox{ if $n-t$ is even}\\ \\
          -\frac{(n+1)(t-1)}{2} & \mbox{ if $n-t$ is odd}
        \end{array}
        \right.
     \end{array}$$
   \end{Theorem}

   In this paper we consider the case of the ideals of minors of size $t=n-2$. As a special case of \ref{pr:basic} we have:

  \begin{Corollary}
    \label{co:n-2} The ring $S/I_{n-2}$ is a Gorenstein normal domain.
    Its dimension, multiplicity, regularity and $a$-invariant are:
    $$
    \begin{array}{llll}
      \dim(S/I_{n-2})=\frac{(n+4)(n-3)}{2} & & &e(S/I_{n-2})={n+2 \choose 6} +{n+3 \choose 6} \\ \\
      \reg(S/I_{n-2})=2(n-3)  &  & &a(S/I_{n-2})=-\frac{n(n-3)}{2}.
    \end{array}
    $$
  \end{Corollary}

  \begin{Remark}
    \label{compressed}
    Let $A$ be a Gorenstein graded $K$-algebra of even Castelnuovo-Mumford regularity $2s$ and codimension $c$.
    Then the $h$-polynomial $h(z)=\sum_{i=0}^{2s} h_iz^i$ satisfies the inequality $h_i\leq \binom{c-1+i}{c-1}$
    for $i=0,\dots,s$. Since $h_{2s-i}=h_i$, this gives an upper bound for the multiplicity of $A$ only in terms
    of $c$ and $s$:
    $$e(A)\leq \binom{c-1+s}{c}+\binom{c+s}{c}$$
    and $A$ is said to be compressed if  the equality holds. In other words, $A$ is compressed if its h-polynomial is given by
    $$\sum_{i=0}^{s} \binom{c-1+i}{c-1} z^i+\sum_{i=0}^{s-1} \binom{c-1+i}{c-1} z^{2s-i}.$$
    A simple computation shows that the minimal free resolution of a compressed Gorenstein $K$-algebra of even regularity is pure
    (i.e. ~only one shift in each homological position) and hence its Betti numbers just depend on $s$ and $c$. Explicit expressions
    for the Betti numbers can be worked out, they can be found for example in \cite{Boji,Iarrobino,Schenzel}.
  \end{Remark}

  We see from \ref{co:n-2} that $S/I_{n-2}$ is compressed of even regularity.
  We obtain the following expressions for the $h$-polynomial and the Betti numbers of $S/I_{n-2}$.

  \begin{Proposition}
    \label{betti}
    The codimension of $S/I_{n-2}$ is $6$, its $h$-polynomial is
    $$\sum_{i=0}^{n-3} \binom{ 5+i}{5} z^i+\sum_{i=0}^{n-4} \binom{ 5+i}{5} z^{2(n-3)-i}$$
    and its non-zero Betti numbers are: $\beta_{00}=\beta_{6,2n}=1$ and
    $$\beta_{6-i,n+3-i}=\beta_{i,n-3+i}=
      \left\{
      \begin{array}{ll}
        \frac{(n+1)n^2(n-1)}{12}  & \mbox{if \   }i=1\\ \\
        \frac{(n+2)n^2(n-2)}{3}    & \mbox{if \   } i=2  \\ \\
        \frac{(n+2)(n+1)(n-1)(n-2)}{2} & \mbox{if \   } i=3
      \end{array}\  \  .
      \right.
    $$
  \end{Proposition}

 \section{The choice of the leading terms}
 \label{choice of leading terms}

   We define a term order on the monomials in $S$:

   \medskip

   Set \begin{eqnarray*}
      V & = & \{ x_{ij} ~|~1 \leq i \leq j \leq n \}, \\
      D & = & \{ x_{ii}~ |~1 \leq i \leq n\}, \\
      U & = & \{ x_{13},x_{24},\ldots, x_{n-2,n} \} \cup \{ x_{12}, x_{n-1,n} \}.
   \end{eqnarray*}

   Consider a termorder $\prec$  given by the reverse
   lexicographic order on $V$ with the variables ordered as follows:
   $$\underbrace{ x_{11} \succ x_{22} \succ \cdots \succ x_{nn}}_{D}\underbrace{\succ x_{13}
     \succ \cdots \succ x_{n-2,n} \succ x_{12} \succ x_{n-1,n}}_U \succ \mbox{all the other variables}. $$

   In the sequel, for $s$-elements subsets $\{\alpha_1,\ldots, \alpha_s\}$ and $\{\beta_1,\ldots, \beta_s\}$
   of $\{ 1, \ldots, n\}$ we write $[\alpha_1,\ldots, \alpha_s|\beta_1,\ldots, \beta_s]$ for the minor
   of the generic symmetric matrix $X$ defined by selecting rows $\{\alpha_1,\ldots, \alpha_s\}$ and 
   columns $\{\beta_1,\ldots, \beta_s\}$ in the given order. Note that we sometimes speak of the row or
   column of a minor by which we mean the row or column of the submatrix selected to compute the minor.  

   For classifying the leading monomials of the $(n-2)$-minors we first need some preparatory lemmas.

  \begin{Lemma}
     \label{lem:small}
     Let $1 \leq s \leq n-2$ and $\a_1 < \cdots < \a_s$ and $\b_1 < \cdots < \b_s$
     be two sequences of distinct indices such that $x_{\a_\ell,\b_\ell} \in U$ for $1 \leq \ell \leq s$.
     Then
     $$\iin_\prec [\a_1,\ldots \a_s | \b_1,\ldots, \b_s] = [\a_{1}|\b_{1}] \cdots [\a_s|\b_s].$$  
  \end{Lemma}
  \begin{proof}
    We prove the assertion by induction on $s$. The case $s =1$ is trivial and we can assume $s > 1$.

    Set $M=[\a_1,\ldots,\a_s|\b_1 ,\ldots, \b_s]$. 
    We expand $M$ along the $s$\textsuperscript{th} column:
    $$M=\sum_{i=1}^{s} \pm[\a_i \,|\,\b_s][\a_1,\ldots, \a_{i-1}, \a_{i+1} ,\ldots, \a_s|\b_1,\ldots, \b_{s-1}].$$
    By induction we have that $$\iin_\prec( [\a_1,\ldots,\a_{s-1} \, |\, \b_1, \ldots,\b_{s-1}])=x_{\a_1\b_1} \cdots x_{\a_{s-1} \b_{s-1}}.$$

    By $\a_1 < \cdots < \a_s < \b_{s}$ it follows that $\a_i \leq \b_s -2$ for
    $1 \leq i \leq s-1$ and equality can only hold if $i = s-1$ and $(\a_s,\b_s) = (n-1,n)$ and $(\a_{s-1},b_{s-1}) = (n-2,n)$.
    But this contradicts $\b_1 < \cdots < \b_s$ and hence $\a_i < \b_s -2$ for $1 \leq i \leq s-1$.
    This implies that $[\a_i\,|\,\b_s]\not\in D\cup U$ for every $i=1,\ldots,s-1$.
    Since $[\a_s|\b_s] \in U$ and $\iin_\prec( [\a_1,\ldots,\a_{s-1} \, |\, \b_1, \ldots,\b_{s-1}])$ is a monomial in $U$
    it follows from the choice of $\prec$ that $\iin_\prec(M)$ takes the desired form.
  \end{proof}

  \begin{Lemma}
     \label{lem:newlittle}
     Let $1 \leq r \leq s \leq n-2$. Assume we have two sequences 
     $$a_1,\ldots, a_r, \a_{r+1} , \cdots , \a_s \mbox{~and~} b_1,\ldots, b_r, \b_{r+1} , \cdots , \b_s$$
     of  distinct indices such that: 
     \begin{itemize}
        \item $x_{\a_\ell,\b_\ell} \in U$ for $r+1 \leq \ell \leq s$,
        \item $a_i < \b_j$ and $x_{a_i, \b_j} \not\in D\cup U$ for $1 \leq i \leq r$ and $r+1 \leq j \leq s$,
        \item $x_{\a_i, \b_j} \not\in D\cup U$ for $r+1 \leq i < j \leq s$,
        \item $\iin_\prec [a_1,\ldots, a_r|b_1 , \ldots, b_r]$ is a monomial in $U$.  
     \end{itemize}
     Then
     \begin{eqnarray*} 
       \iin_\prec [a_1,\ldots ,a_r,\a_{r+1} , \ldots \a_s | b_1,\ldots, b_r,\b_{r+1},\ldots, \b_s] & = & \\
       \iin_\prec [a_1,\ldots, a_r|b_1,\ldots, b_r] [\a_{r+1}|\b_{r+1}] \cdots [\a_s|\b_s]. &   & 
     \end{eqnarray*} 
  \end{Lemma}
\begin{proof}
      We proceed by induction on $s-r$. 
      If $s = r$ then the assertion is trivially true. 
      Assume $r < s$. 
      Set $M = [a_1,\ldots ,a_r,\a_{r+1} , \ldots, \a_s | b_1,\ldots, b_r,\b_{r+1},\ldots, \b_s]$.
      We expand $M$ along the $s$\textsuperscript{th} column:
      \begin{eqnarray*} M & = & \sum_{i=1}^{r} \pm[a_i \,|\,\b_s][a_1,\ldots, a_{i-1}, a_{i+1} ,\ldots,a_r, \a_{r+1},\ldots, \a_s|b_1,\ldots, b_r, \b_{r+1},\ldots, \b_{s-1}]+ \\
                          &   & \sum_{i=r+1}^{s} \pm[\a_i \,|\,\b_s][a_1,\ldots, a_r, \a_{r+1}, \ldots,\a_{i-1},\a_{i+1} ,\ldots, \a_s|b_1,\ldots, b_r, \b_{r+1},\ldots, \b_{s-1}].
      \end{eqnarray*}
      For the summation index $i = s$ by induction we have
      \begin{eqnarray*} 
        \iin_\prec ( [\a_s \,|\,\b_s][a_1,\ldots, a_r, \a_{r+1}, \ldots,\a_{s-1}|b_1,\ldots , b_r, \b_{r+1},\ldots, \b_{s-1}] )& =  & \\
        \iin_\prec ([\a_s|\b_s] [a_1,\ldots, a_r|b_1,\ldots , b_r] [\a_{r+1}|\b_{r+1}] \cdots [\a_{s-1}|\b_{s-1}]).& & 
      \end{eqnarray*}
 
      which by our assumption on $\iin_\prec ([a_1,\ldots, a_r|b_1,\ldots , b_r])$ is a monomial in $U$.
      Note that $[a_i|\b_s]$ do not lie in $D \cup U$ 
     for $1 \leq i \leq r$, and  $[\a_i,\b_s]$ also do not lie in $D \cup U$ for $r+1 \leq i \leq s-1$. 
      Thus by the choice of $\prec$ it follows that 
      \begin{eqnarray*} 
         \iin_\prec (M) & = & [\a_s \,|\,\b_s][a_1,\ldots, a_r, \a_{r+1}, \ldots,\a_{s-1}|b_1,\ldots , b_r, \b_{r+1},\ldots, \b_{s-1}] \\
                      & = & \iin_\prec ([a_1,\ldots, a_r|b_1,\ldots, b_r] )[\a_{r+1}|\b_{r+1}] \cdots [\a_s|\b_s].
      \end{eqnarray*}
  \end{proof}
 
  The same proof but expanding along the first row yields:

  \begin{Lemma}
     \label{lem:newlittle2}
     Let $1 \leq q \leq s \leq n-2$. Assume we have two sequences $$\a_1,\ldots, \a_{s-q} ,a_{s-q+1} , \dots , a_s 
     \mbox{~and~} \b_1,\ldots, \b_{s-q}, b_{s-q+1} , \cdots , b_s$$
     of  distinct indices such that: 
     \begin{itemize}
        \item $x_{\a_\ell,\b_\ell} \in U$ for $1 \leq \ell \leq s-q$,
        \item $\a_i < b_j$ and $x_{\a_i, b_j} \not\in D\cup U$ for $1\leq i \leq s-q$ and $s-q+1 \leq j \leq s$,
        \item $x_{\a_i, \b_j} \not\in D\cup U$ for $1 \leq i < j \leq s-q$,
        \item $\iin_\prec [a_{s-q+1} , \dots , a_s|b_{s-q+1} , \cdots , b_s]$ is a monomial in $U$.  
     \end{itemize}
     Then
     \begin{eqnarray*} 
     \iin_\prec [\a_1,\ldots, \a_{s-q},a_{s-q+1},\ldots, a_s|\b_1,\ldots, \b_{s-q},b_{s-q+1},\ldots, b_s] & = & \\ 
        \,[\a_{1}|\b_{1}] \cdots [\a_{s-q}|\b_{s-q}] \iin_\prec [a_{s-q+1},\ldots, a_s|b_{s-q+1},\ldots, b_s]. &   & 
     \end{eqnarray*} 
  \end{Lemma}

  \begin{Lemma}
     \label{lem:special}
     Let $a_1 = 1$, $a_{2i} = 2i+1$, $a_{2i+1} = 2i$ for $1 \leq i \leq \ell$ and
         $b_1 = 2$, $b_{2i} = 2i-1$, $b_{2i+1} = 2i+2$ for $1 \leq i \leq \ell$.
     Then for $1 \leq \ell \leq \frac{n-2}{2}$  one has
     \begin{eqnarray} 
        \label{eq:odd}
        \iin_\prec [a_1,\ldots, a_{2\ell+1}|b_1,\ldots, b_{2\ell+1}] & = & 
           x_{1,2} x_{1,3} x_{2,4} \cdots x_{2\ell,2\ell+2}.
     \end{eqnarray} 
     and for $1 \leq \ell \leq \frac{n-3}{2}$  one has
     \begin{eqnarray} 
       \label{eq:even}
       \iin_\prec [a_1,\ldots, a_{2\ell}|b_1,\ldots, b_{2\ell}] & = &  
         x_{1,2} x_{1,3} x_{2,4} \cdots x_{2\ell-1,2\ell+1}.
     \end{eqnarray}
  \end{Lemma}
  \begin{proof}
    We proceed by induction on $\ell$.
    For $\ell = 1$ a direct computation yields the result.     
    Assume $\ell > 1$. 

    First, we consider $[a_1,\ldots, a_{2\ell}|b_1,\ldots, b_{2\ell}]$ for $1 \leq \ell \leq \frac{n-3}{2}$. 
    The entries of the $(2\ell)$\textsuperscript{th} row of $[a_1,\ldots, a_{2\ell}|b_1,\ldots, b_{2\ell}]$
    are $[2\ell+1|2]$, $[2\ell+1|2i-1]$ for $1 \leq i \leq \ell$ and $[2\ell+1|2i+2]$ for $1 \leq i \leq \ell-1$ .
    Since $1 < \ell$ and $2\ell+1 \leq n-2$ we have that $[2\ell+1|2]$, $[2\ell+1|2i+2] \not\in D \cup U$ for $1 \leq i \leq \ell-1$ .  
    For $1 \leq i \leq \ell-1$ also $[2\ell+1|2i-1] \not \in D \cup U$. 
    Thus expanding  $[a_1,\ldots, a_{2\ell}|b_1,\ldots, b_{2\ell}]$ along its last row we obtain 
    for all columns except for the $(2\ell)$\textsuperscript{th} column a factor that does not lie $D \cup U$.
    For the last column we get the term
    $$[2\ell+1|2\ell-1] [a_1,\ldots, a_{2(\ell-1)+1}|b_1,\ldots, b_{2(\ell-1)+1}].$$
    By the induction hypothesis \eqref{eq:odd} its initial term is a monomial in $U$ and thus it is the initial term 
    $[a_1,\ldots, a_{2\ell}|b_1,\ldots, b_{2\ell}]$. 
    This implies \eqref{eq:even} for $\ell$.

    Analogously, consider the $(2\ell+1)$\textsuperscript{st} column of $[a_1,\ldots, a_{2\ell+1}|b_1,\ldots, b_{2\ell+1}]$.
    Its entries are $[1|2\ell+2]$, $[2i+1|2\ell+2]$, $i = 1,\ldots, \ell$ and $[2i|2\ell+2]$, $i=1,\ldots, \ell$.
    Since $1 < \ell$ and $2\ell+1 \leq n-2$ we have that $[1|2\ell+2]$, $[2i+1|2\ell+2] \not \in D \cup U$ for $i=1,\ldots, \ell$.  
    For $i = 1,\ldots, \ell-1$ also $[2i|2\ell+2] \not\in D \cup U$. 
    Thus expanding $[a_1,\ldots, a_{2\ell+1}|b_1,\ldots, b_{2\ell+1}]$ along its last column we obtain 
    for all row except of the $(2\ell+1)$\textsuperscript{st} row a factor that does not lie $D \cup U$.
    For the last row we get the term
    $$[2\ell|2\ell+2] [a_1,\ldots, a_{2\ell}|b_1,\ldots, b_{2\ell}].$$
    By \eqref{eq:even} for $\ell$ its initial term is a monomial in $U$ and therefore is the initial term of 
    $[a_1,\ldots, a_{2\ell+1}|b_1,\ldots, b_{2\ell+1}]$. 
    This then implies \eqref{eq:odd} for $\ell$.
  \end{proof}

  \begin{Lemma}
     \label{lem:special2}
     Let $a_{s} = n-1$, $a_{s-2i} = n-(2i+1)$, $a_{s-(2i-1)} = n-(2i-2)$ for $1 \leq i \leq \ell$ and
         $b_{s} = n$, $b_{s-2i} = n-(2i-1)$, $b_{s-(2i-1)} = n-2i$ for $1 \leq i \leq \ell$.
     Then for $1 \leq \ell \leq \frac{n-3}{2}$  
     \begin{eqnarray} 
        \iin_\prec [a_{s-2\ell},\ldots, a_{s}|b_{s-2\ell},\ldots, b_{s}] & = & 
           x_{n-2\ell-1,n-2\ell+1} \cdots x_{n-3,n-1} x_{n-2,n} x_{n-1,n}
     \end{eqnarray} 
     and for $1 \leq \ell \leq \frac{n-2}{2}$  
     \begin{eqnarray} 
       \iin_\prec [a_{s-2\ell+1} ,\ldots, a_{s}|b_{s-2\ell+1} , \ldots,  b_{s}] & = &  
          x_{n-2\ell,n-2\ell+1} \cdots x_{n-3,n-1} x_{n-2,n} x_{n-1,n}.
     \end{eqnarray}
  \end{Lemma}
    
  \begin{Proposition}
     \label{le:initial1}
     For a number $1 \leq s \leq n-2$ let $1 \leq \a_1 \leq \a_2 <\cdots < \a_s \leq n-1$, $2 \leq \b_1 < \cdots < \b_{s-1} \leq \b_s \leq n$ be indices such that
     $m=x_{\a_1 \b_1}\cdots x_{\a_s \b_s}$ is a squarefree monomial with
     $x_{\a_t \b_t}\in U$ for every $t=1,\ldots ,s$.
     Then $m$ is the leading term of an $s$-minor of $X$ with respect to $\prec$.
  \end{Proposition}
  \begin{proof}
     We distinguish cases according to the order relations among the $\a_i$ and among the $\b_i$.

     \noindent {\sf Case 1:} $\a_1 < \cdots < \a_s$ and $\b_1 < \cdots < \b_s$. 

     Here the assertion follows from \ref{lem:small}.

     \noindent {\sf Case 2:} $\a_1 \leq \cdots \leq \a_s$ with at least one equality and $\b_1 < \cdots < \b_s$. 

     In this situation it follows that $1 = \a_1 = \a_2 < \cdots < \a_s$ and $\b_1 = 2$, $\b_2 = 3$.  
     If $\a_i = i-1$ for $i=2,\ldots ,s$ then by $s \leq n-2$ we must have $\b_i = i+2$ and the 
     assertion follows from \ref{lem:special}.  
     Thus we can assume that there is a $2 \leq r \leq s-1$  such that $\a_{r+1} \neq r$. 
     We choose $r$ minimal with this property.
     Set $m' = x_{\a_1\b_1}x_{\a_2\b_2}\cdots x_{\a_r\b_r}=
     x_{12}x_{13}x_{24} \cdots x_{r-1,r+1}$. By \ref{lem:special} there exist
     $a_1,\ldots, a_r$, $b_1,\ldots, b_r$ such that $\iin_\prec [a_1,\ldots, a_r|b_1,\ldots, b_r] =
     m'$. 
     By the \ref{lem:special} after possibly exchanging rows and columns 
     we can assume $a_i \leq r$ for $i=1,\ldots, r$. Note that exchanging rows and columns does not change the minor since $X$ is
     symmetric. Then by the choice of $r$ we have $r+2 \leq \a_{r+1}+1 < \b_{r+1} < \cdots < \b_s$, thus $a_i<\b_j-2$ and 
     $[a_i | \b_j]\not\in D\cup U$ for $i=1,\ldots, r$, $j=r+1,\ldots,s$. 
     Moreover, by $\a_i < \a_j = \b_j-2$ for $r+1 \leq i < j \leq s$ it follows
     that $x_{\a_i\b_j} \not\in D \cup U$ for $r+1 \leq i < j \leq s$. 
     Therefore, we can apply \ref{lem:newlittle} to 
     $a_1,\ldots, a_r,\a_{r+1},\ldots, \a_s$ and $b_1,\ldots, b_r,\b_{r+1},\ldots, \b_s$. 
     This shows that 
     \begin{eqnarray*} 
        \iin_\prec [a_1,\ldots, a_r,\a_{r+1},\ldots, \a_s|b_1,\ldots, b_r,\b_{r+1},\ldots, \b_s] & = & \\ 
        \iin_\prec [a_1,\ldots, a_r|b_1,\ldots, b_r] [\a_{r+1}|\b_{r+1}] \cdots [\a_s|\b_s] & = & \\
        m' x_{\a_{r+1} \b_{r+1}} \cdots x_{\a_s\b_s} & = & \\
        x_{\a_1\b_1}x_{\a_2\b_2}\cdots x_{\a_r\b_r}x_{\a_{r+1} \b_{r+1}} \cdots x_{\a_s\b_s}  &=& m.
     \end{eqnarray*}

     \noindent {\sf Case 3:} $\a_1 < \cdots < \a_s$ and $\b_1 \leq \cdots \leq \b_s$ with at least one equality. 

     In this situation it follows that $\b_1 < \cdots < \b_{s-1} = \b_s = n$ and $\a_{s-1} = n-2$, $\a_s = n-1$.  
     If $\b_{s-q} = n-q+1$ for $q=1,\ldots,s-1$ then by $s \leq n-2$ we must have $\a_{s-q} = n-q-1$ for $q=0,\ldots, s-1$
     and the assertion follows from \ref{lem:special2}.  
     Thus we can assume that there is a $2 \leq q \leq s-1$  such that $\b_{s-q} \neq n-q+1$. 
     We choose $q$ minimal with this property.
     Set $m' = x_{\a_{s-q+1}\b_{s-q+1}}x_{\a_{s-q+2}\b_{s-q+2}}\cdots x_{\a_s\b_s}=
     x_{n-q,n-q+2} \cdots x_{n-3,n-1} x_{n-2,n}x_{n-1,n}$. By \ref{lem:special2} there exist
     $a_{s-q+1} ,\ldots, a_s$, $b_{s-q+1},\ldots, b_s$ such that $\iin_\prec [a_{s-q+1},\ldots, a_s|b_{s-q+1},\ldots, b_s] =
     m'$. 
     By the \ref{lem:special} after possibly exchanging rows and columns 
     we can assume $b_i \geq n-q+1$ for $i=s-q+1,\ldots, s$. Then by the choice of $q$ we have 
     $\a_1 < \cdots < \a_{s-q} \leq \b_{s-q} -2 \le n-q-2$, thus $b_i>\a_j-2$ and $[\a_j | b_i]\not\in D\cup U$ for $i=s-q+1,\ldots, s$ and $j=1,\ldots,s-q$. 
     Moreover, by $\a_i < \a_j = \b_j-2$ for $1 \leq i < j \leq s-q$ it follows
     that $x_{\a_i\b_j} \not\in D \cup U$ for $1 \leq i < j \leq s-q$. 
     Therefore,
     we can apply \ref{lem:newlittle2} to 
     $\a_1,\ldots, \a_{s-q},a_{s-q+1},\ldots, a_s$ and $\b_1,\ldots, \b_{s-q} ,b_{s-q+1},\ldots, b_s$. 
     This shows that 
     \begin{eqnarray*} 
        \iin_\prec [\a_1,\ldots, \a_{s-q},a_{s-q+1},\ldots, a_s|\b_1,\ldots, \b_{s-q},b_{s-q+1},\ldots, b_s] & = & \\ 
        \,[\a_{1}|\b_{1}] \cdots [\a_{s-q}|\b_{s-q}] \iin_\prec [a_{s-q+1},\ldots, a_s|b_{s-q+1},\ldots, b_s] & = & \\ 
        x_{\a_{1} \b_{1}} \cdots x_{\a_{s-q}\b_{s-q}} m' & = & \\
        x_{\a_1,\b_1} \cdots x_{\a_{s-q},\b_{s-q}}x_{\a_{s-q+1}\b_{s-q+1}}x_{\a_{s-q+2}\b_{s-q+2}}\cdots x_{\a_s\b_s} &= &m .
     \end{eqnarray*}

     \noindent {\sf Case 4:} $\a_1 \leq \cdots \leq \a_s$ and $\b_1 \leq \cdots \leq \b_s$ with at least one equality in both. 

     In this situation it follows that $1 = \a_1 = \a_2 < \cdots < \a_s$ and $\b_1 = 2$, $\b_2 = 3$  
     and $\b_1 < \cdots < \b_{s-1} = \b_s = n$ and $\a_{s-1} = n-2$, $\a_s = n-1$.  
     Thus we can assume that there is a $2 \leq r \leq s-1$ such that $\a_{r+1} > r$
     and there is a $2 \leq q \leq s-1$  such that $\b_{s-q} < n-q+1$. 
     We choose $r$ and $q$ minimal with this property. 
     Set $m' = x_{\a_1\b_1}x_{\a_2\b_2}\cdots x_{\a_r\b_r}=
     x_{12}x_{13}x_{24} \cdots x_{r-1,r+1}$. By \ref{lem:special} there exist
     $a_1,\ldots, a_r$, $b_1,\ldots, b_r$ such that $\iin_\prec [a_1,\ldots, a_r|b_1,\ldots, b_r] =
     m'$. After possibly exchanging rows and columns 
     we can assume $a_i \leq r$ for $i=1,\ldots, r$.
     
     Moreover set $m'' = x_{\a_{s-q+1}\b_{s-q+1}}\cdots x_{\a_s\b_s}=
     x_{n-q,n-q+2} \cdots x_{n-3,n-1} x_{n-2,n}x_{n-1,n}$. 
     
     By \ref{lem:special2} there exist
     $a_{s-q+1} ,\ldots, a_s$, $b_{s-q+1},\ldots, b_s$ such that $$\iin_\prec [a_{s-q+1},\ldots, a_s|b_{s-q+1},\ldots, b_s] =
     m'',$$ and after possibly exchanging rows and columns 
     we can assume $b_i \geq n-q+1$ for $i=s-q+1,\ldots, s$.
     Set $$M= [a_1,\ldots, a_r, \a_{r+1},\ldots, \a_{s-q},a_{s-q+1},\ldots, a_s|b_1,\ldots, b_r,\b_{r+1},\ldots, \b_{s-q},b_{s-q+1},\ldots, b_s].$$ 
     Note that the row and the column indices of $M$ are distinct since  $a_i\le r<\a_{r+1}$, $b_i\le r+1<\b_{r+1}$ for $i=1,\ldots r$ and
     $\a_{s-q}<n-q-1\le a_j$,  $\b_{s-q}<n-q+1\le b_j$ for $j=s-q+1,\ldots, s$.
    
     We prove that $\iin_\prec(M)=m$, by induction on $r\ge 2$.

     If $r=2$, then $M =[2,1, \a_{3},\ldots, \a_{s-q},a_{s-q+1},\ldots, a_s|1,3,\b_{3},\ldots, \b_{s-q},b_{s-q+1},\ldots, b_s]$.
     By expanding along the first row one obtains:
     $$M = [2 \,|\,1]M_1- [2 \,|\,3]M_2+ \sum_{j=3}^{s-q} \pm[2 \,|\,\b_j]M_j + \sum_{j=s-q+1}^{s} \pm[2 \,|\,b_j]N_j . $$ 
        Note, that $[2,3]\not\in D\cup U$ and since $\b_j\ge 5$ and $b_j>5$, then  $[2 \,|\,\b_j]$ and $[2 \,|\,b_j]$ do not lie in $D\cup U$ for every $j$.
     Thus $\iin(M)=\iin([2,1]M_1)=[2,1]\iin(M_1)$. 
     Indeed, the only element from $D \cup U$ in the first row of $M_1$ is i$[1,3]$. Hence, by expanding $M_1$ along the first row and by using Case 3 one obtains:
     $$\iin(M_1)=[1,3]\iin([\a_{3},\ldots, \a_{s-q},a_{s-q+1},\ldots, a_s|\b_{3},\ldots, \b_{s-q},b_{s-q+1},\ldots, b_s])=m.$$
     \medskip

     Now assume $r>2$.  

     We consider the case when $r$ is even. The case when $r$ odd can be treated in the same way. 
    
     One has:
     $$M= [2,1,a_3, \ldots, a_r, \a_{r+1},\ldots, \a_{s-q},a_{s-q+1},\ldots, a_s|1,3,b_3,\ldots, b_r,\b_{r+1},\ldots, \b_{s-q},b_{s-q+1},\ldots, b_s]$$
     with $a_r=r-1, b_r=r+1,\ a_{r-1}=r,\  b_{r-1}= r-2,\ a_{r-2}=r-3,\ b_{r-2}=r-1,\ b_{r-4}=r-3$. Note that $\a_j\ge r+1$,  $\b_j\ge r+3$ for every $j$ and $b_i>\b_j$ for all $i,j$.

     By expanding $M$ along the $a_r$\textsuperscript{th} row one obtains:
     $$M = [a_r \,|\,1]M_1- [a_r \,|\,3]M_2+ \sum_{j=3}^{r} \pm[a_r \,|\,b_j]P_j+\sum_{j=r+1}^{s-q}\pm[a_r \,|\,\b_j]M_j+\sum_{j=s-q+1}^{s}\pm[a_r \,|\,b_j]N_j $$ 
     
     Consider $P_r=[2,1, a_{3},\ldots, a_{r-1}, \a_{r+1},\ldots, a_s|1,3,b_{3},\ldots,b_{r-1},\b_{r+1},\ldots,\b_{s-q},b_{s-q+1},\ldots, b_s]$. By the assumption 
     on the indices involved in $M$, the only element of the $a_{r-1}$\textsuperscript{st} row in $D\cup U$ is $[a_{r-1},b_{r-1}]=[r,r-2]$. Thus by expanding 
     $P_r$ along the $a_{r-1}$\textsuperscript{st} row we get that $m = \iin([a_r,b_r]P_r)$ equals 

     $$[a_r,b_r][a_{r-1},b_{r-1}][2,1, a_{3},\ldots, a_{r-2}, \a_{r+1},\ldots, a_s|1,3,b_{3},\ldots,b_{r-2},\b_{r+1},\ldots,\b_{s-q},b_{s-q+1},\ldots, b_s]$$ 

     by the induction hypotheses. 

     Now $[a_r \,|\,\b_j]\not\in  D\cup U$ for $j=r+1,\ldots, s-q$ and $[a_r \,|\,b_j]\not\in  D\cup U$ for $j=s-q+1,\ldots, s$. Moreover, if $j=3,\ldots r$  
     the only $[a_r \,|\,b_j]$ that can be in $D\cup U$ are from the set $\{[a_r,b_r]=[r-1,r+1],\  [a_r,b_{r-2}]=[r-1,r-1],\  [a_r,b_{r-4}]=[r-1,r-3]\}$.  

     It follows that
     $$\begin{array}{rl} 
         \iin(M)&=  \iin(m+[a_r \,|\,1]M_1+ [a_r \,|\,3]M_2+[a_r \,|\,b_{r-2}]P_{r-2}+[a_r \,|\,b_{r-4}]P_{r-4})\\
    \           &=  \iin(m+[a_r \,|\,1]M_1+ [a_r \,|\,3]M_2+[a_r \,|\,a_{r}]P_{r-2}+[a_r \,|\,a_r-2]P_{r-4}).
        \end{array}
     $$

     Note, that since $r > 2$ even implies $r\ge 4$. 
     If $r=4$, then $[a_r \,|\,1]=[3,1]$ and $[a_r \,|\,3]=[3,3]=[a_r,b_{r-2}]$. If $r=6$, then $ [a_r \,|\,1]=[5,1]\not\in D\cup U$ and $[a_r \,|\,3]=[a_r,b_{r-4}]$. 
     If $r\ge 8$, then  $a_r\ge 7$ and $ [a_r \,|\,1]$ and $[a_r \,|\,3]$ do not lie in $D\cup U$. Therefore,
     $$\iin(M)=\left\{\begin{array}{ll}
                       \iin(m+[a_r \,|\,a_{r}]P_{r-2}+[a_r \,|\,a_r-2]P_{r-4})& \mbox{ if } r\ge 6\\ 
                       \iin(m+ [3 \,|\, 1]M_1+[3\,|\, 3]P_{r-2}) & \mbox{ if } r=4  \\ 
                    \end{array}. \right.$$

     In the following we prove that $\iin( P_{r-2}),\ \iin(P_{r-4})$ and $\iin(M_1)$ involve at least one indeterminate not in $D\cup U$, 
     thus $[a_r \,|\,a_{r}]\iin(P_{r-2}), [a_r \,|\,a_r-2]\iin(P_{r-4})$ and $[3 \,|\, 1]\iin(M_1)$ are larger than $m$. From this the assertion follows.

     We will only treat the case of $P_{r-2}$, an analogous reasoning will covers the cases $P_{r-4}$ and $M_1$.

     For $P_{r-2}=[2,1, a_{3},\ldots, a_{r-1}, \a_{r+1},\ldots, a_s|1,3,b_{3},\ldots,b_{r-3},b_{r-1},b_r,\b_{r+1},\ldots, b_s]$
     the only element of the $a_{r-1}$\textsuperscript{st} row which is in $D\cup U$ is $[a_{r-1},b_{r-1}]=[r, r-2]$. Expanding along the $a_{r-1}$\textsuperscript{st} row, we get 
     $\iin(P_{r-2})=[a_{r-1},b_{r-1}]Q$, with $$Q=[2,1, a_{3},\ldots, a_{r-2}, \a_{r+1},\ldots, a_s|1,3,b_{3},\ldots,b_{r-3},b_{r-1},b_r,\b_{r+1},\ldots, b_s].$$ Now the only elements 
     in the $a_{r-2}$\textsuperscript{nd} row of $Q$ which are in $D\cup U$ are
     $[a_{r-2}, b_{r-4}  ]=[r-3, r-3]$ and $[a_{r-2}, b_{r-6}  ]=[r-3, r-5]$. Thus by expanding along the $a_{r-2}$\textsuperscript{nd} row we obtain
     $\iin(Q)=\iin([a_{r-2}, b_{r-4}]Q_1+[a_{r-2}, b_{r-6}]Q_2)$ with 

     $$Q_1=[2,1, a_{3},\ldots, a_{r-3}, \a_{r+1},\ldots, a_s|1,3,b_{3},\ldots,b_{r-5},b_{r-3},b_{r-1},b_r,\b_{r+1},\ldots, b_s]$$ and 

     $$Q_2=[2,1, a_{3},\ldots, a_{r-3}, \a_{r+1},\ldots, a_s|1,3,b_{3},\ldots,b_{r-7},b_{r-5},b_{r-4},b_{r-3},b_{r-1},b_r,\b_{r+1},\ldots, b_s].$$ 

     Expanding all the minors we obtain in this way,  step by step along the $a_{r-j}$\textsuperscript{th} rows, with $j=3,...,r-3$, and arguing in the same way, we get 
     $\iin(P_{r-2})=\iin([2,1| c, b_r]p)$ with $p$ a polynomial involving only indeterminates in $D\cup U$, and $c<b_r$. Thus  
     $\iin(P_{r-2})=\iin([2,c][1,b_r]p-[2,b_r][1|c]p)$ which is bigger than $m$ since $[1,b_r]$ and $[2,b_r]$ are not in $D\cup U$.
     \medskip

\end{proof}

  Now we are in position to give a description of some leading terms of $s$-minors of $X$ with respect to $\prec$.

  \begin{Proposition}
    \label{pr:initmon}
    Let $m$ be a square-free monomial of degree $s$ for some $1 \leq s \leq n-2$ in the set $D \cup U$ such that for all $x_{ii}|m$ and $x_{hk}|m$
    with $x_{hk}\in U$ one has $i\neq h$ and $i\neq k$. Then $m$ is the leading term of an $s$-minor of $X$ with respect to $\prec$.
  \end{Proposition}
  \begin{proof}
    Let $m=x_{i_1 i_1}\cdots x_{i_j i_j}x_{\a_1 \b_1}\cdots x_{\a_p \b_p}$ be a squarefree monomial with 
    $x_{\a_\ell \b_\ell}\in U$ for every $\ell=1,\ldots, p$ and $j+p=s$.
    We prove by induction on $j$ that if $i_k\not=\a_\ell$ and $i_k\not=\b_\ell$ for every $1 \leq k \leq j$, $1 \leq \ell \leq p$, then
    $m$ is the leading term of an $s$-minor $M$ of the matrix $X$  with respect to $\prec$.

    The induction base $j=0$ is a consequence of \ref{le:initial1} and we may now assume $j>0$.

    We set $m=x_{i_1 i_1}\cdots x_{i_j i_j}\cdot m'$ with $m'=x_{\a_1 \b_1}\cdots x_{\a_p \b_p}$.
    By the \ref{le:initial1} there exists $a_1,\ldots, a_p,b_1,\ldots, b_p$ such that 
    for $M'=[a_1,\ldots,a_p\,|\,b_1,\ldots,b_p]$ we have $\iin_\prec(M')=m'$ and
    $\{ \a_1,\ldots, \a_p ,\b_1,\ldots, \b_p\} = \{a_1,\ldots, a_p,b_1,\ldots,b_p\}$.
    To conclude the proof we show that $m=\iin_\prec (M)$, with
    $$M=[i_1,\ldots,i_j,a_1,\ldots,a_p\,|\,i_1,\ldots,i_j,b_1,\ldots,b_p].$$

    Note that by assumption $i_k\not=a_\ell$ and $i_k\not=b_\ell$ for every $1 \leq k \leq j$,
    $1 \leq \ell \leq p$, thus all the row indices
   (resp. the column indices) in $M$ are distinct and $M\not=0$.

    Expanding $M$ along its first row one has
    $$M= [i_1\,|\, i_1][i_2,\ldots,i_j,a_1,\ldots,a_p\,|\,i_2,\ldots,i_j,b_1,\ldots,b_p]+
    \sum_{k=2}^{j} \pm [i_1 \,|\,i_k]M_k+\sum_{h=1}^{p} \pm [i_1 \,|\,b_h]N_h.$$
    By induction $\iin_\prec([i_2,\ldots,i_j,a_1,\ldots,a_p\,|\,i_2,\ldots,i_j,b_1,\ldots,b_p])=x_{i_2\, i_2}\cdots x_{i_j\, i_j}\cdot m'$,
    thus to conclude we have to prove that in the two sums in the expansion cannot appear any term bigger than $m$.

    \smallskip

    First consider the terms in $ \sum_{k=2}^{j} \pm [i_1 \,|\,i_k]M_k.$

    Let $\mathcal{I}=\{i_1,i_2,\ldots,i_j\}.$ If $i_1+2\not\in \mathcal{I}$ and $(i_1,i_2)\not\in\{(1,2),(n-1,n)\}$, then $x_{i_1\,i_k}\not\in D\cup U$ for every $k$, thus no term bigger than $m$ appears in the sum.

    If $i_1+2\in \mathcal{I}$, then $i_2=i_1+2$ or $i_2=i_1+1$ and $i_3=i_1+2$. Suppose $i_2=i_1+2$ (in the other case one concludes similarly).
    The only possible terms bigger than $m$  come from
    $[i_1 \,|\,i_2]M_2=[i_1 \,|\,i_1+2][i_2,\ldots,i_j,a_1,\ldots, a_p\,|\, i_1,i_3\ldots,i_j,b_1,\ldots,b_p]$.
    If $i_2+2\not\in\I$ we conclude by expanding $M_2$ along its first row.
    Otherwise we can repeat the reasoning until we find $i_h$ such that
    $i_h+2\not\in\I$ and we conclude.
    It remains to consider the cases $(i_1,i_2)=(1,2)$ and $(i_1,i_2)=(n-1,n)$, that can be treated similarly,
    by expanding along the first row and remembering that $x_{12}$ and $x_{n-1,n}$ are the smallest indeterminates in $D\cup U$.

    \smallskip

    Consider now the terms in $\sum_{h=1}^{p} \pm [i_1 \,|\,b_h]N_h.$ The only terms to be considered are the ones with $x_{i_1 \, b_h}\in D\cup U$, that is
    $$x_{i_1 \, b_h}\in\{ x_{12},x_{21}, x_{n-1\, n}, x_{n\, n-1}, x_{k,\, k+2}, x_{k+2,\, k}, \text{ for some }k\in\{1,\ldots,n-2\} \}.$$
    Let start with $x_{i_1 \, b_h}=x_{n-1\, n}$, that is $i_1=n-1$ and $b_h=n$. In particular $b_h=b_p$, then $i_k\not=n$ for every $k$, thus $j=1$ and 
    $M=[n-1,a_1,\ldots,a_p\,|\,n-1,b_1,\ldots,b_{p-1},n].$
    By developing $M$ along its first row one has: $$[n-1\,|\, n-1][a_1,\ldots,a_p\,|\, b_1,\ldots,b_{p-1},n]
    \pm [n-1\,|\,n][a_1,\ldots,a_p\,|\,n-1,b_1,\ldots,b_{p-1}]+ \textrm{ other terms }$$
    all containing a variables not in $D\cup U$; thus the conclusion follows by induction and by the fact that $x_{n-1,n}$ is the smallest variables in $D\cup U$.

    Similarly one concludes in the case $x_{i_1 \, b_h}=x_{n\, n-1}$.

    Consider now $x_{i_1 \, b_h}=x_{12}$, that is $i_1=1, b_h=b_1=2$.
    In particular $\a_1=2$ and $a_1=4$. Note that $x_{13}$ does not appear in the first row of $M$ on the right of $x_{12}$,
    otherwise it would be $\a_t+2=3$ for some $t$ and we would have $\a_t=1=i_1$, that contradicts the hypothesis.
    Now if $i_2\not=3$, the conclusion follows by developing $M$ along the first row and noting that $x_{12}$ is the
    smallest indeterminate in $D\cup U\setminus \{x_{n-1\,n}\}$. If $i_2=3$, developing $M$ along the first row we have to
    consider the term $T=x_{13}[3,i_3,\ldots,i_j,a_1,\ldots,a_p\,|\, 1,i_3, \ldots,i_j,b_1,\ldots,b_p]$. If $i_3\not=5$
    the leading term of $T$ is divided by $x_{13}^2$, thus it is smaller than the term we want to show to be the leading one;
    so we are done. If $i_3=5$ we go on expanding until we find $i_h\not=h+2$ and we conclude in the same way.

    Suppose $x_{i_1 \, b_h}=x_{k,\, k+2}$, for some $k$, that is, $b_h=i_1+2$; $b_h\not\in\{\b_1,\ldots,\b_p\}$,
    otherwise it would be $i_1\in\{\a_1,\ldots,\a_p\}$, that contradicts the hypothesis. Thus we are in one of the Cases 2,3,4 of   
    \ref{le:initial1}. There are then only two possibilities: $(a_1,b_1)$ is equal to $(1,2)$ or
    to $(2,1)$, thus $\{1,2,\ldots,i_1\}\subseteq \{\a_i,\b_i\,|\, i=1,\ldots, p\}$ which contradicts the hypothesis.
    Or $(a_1,b_1)$ is equal to $(n-1,n)$ or to $(n,n-1)$ which leads to a similar contradiction. Thus this cannot occur.
  Analogously, one proves that it cannot be $x_{i_1 \, b_h}=x_{2\,1}$.

    If $x_{i_1 \, b_h}=x_{k+2,\, k}$ for some $k=1,\ldots,n-2$, that is $b_h=i_1-2$, then $b_h\in\{\b_1,\ldots,\b_p\}.$ In fact
    $b_h\in\{\a_1,\ldots,\a_p\}$ would imply $i_1\in \{\b_1,\ldots,\b_p\}$, that contradicts the hypothesis.
    One concludes also in this case arguing as in the case $x_{i_1 \, b_h}=x_{12}$, and this concludes the proof.
  \end{proof}

\section{The initial complex}
\label{r-covered}

 We introduce some notions that will be used to describe the simplicial complex associated to $\iin_\prec (I_{n-2})$.
Denote by $C(m,d)$  the boundary complex of the $d$-dimensional cyclic polytope with $m$ vertices.
Recall that, by Gale's evenness condition, the facets of 
 $C(m,d)$  are  the subsets $M \subseteq [m]$ of size $d$ such that for any
  two $i,j \in [m] \setminus M$ with $i < j$ the number of elements $\ell \in M$ 
  for which $i < \ell < j$ is even (see \cite[Theorem 0.7]{Ziegler}).

 Moreover let $C_m = ([m],E)$ be the $m$-cycle graph on vertex set $[m] = \{1,\ldots, m\}$
  and edge set $E = \{ \{i,i+1\} | 1 \leq i \leq m-1\} \cup \{\{1,m\}\}$.
  For some $r < \frac{m}{2}$ consider the simplicial complex $\cM_{m,r}$ consisting of all subsets $M$ 
  of $[m]$ such that the vertices in $M$ are covered by a (partial) matching of $C_m$ of size $r$.  

  \begin{Lemma}
     \label{le:cyclic}
     The simplicial complex $\cM_{m,r}$ is the boundary complex of the cyclic polytope
     $C(m,2r)$.
     In particular, $\cM_{m,r}$ is
     pure of dimension $2r-1$ with
     ${m-r \choose r}+{m-r-1 \choose r-1}$ facets.
  \end{Lemma}
  \begin{proof}
    First we note that by definition $\cM_{m,r}$ is a pure simplicial complex and that
    the boundary complex of $C(m,2r)$ is pure as a boundary complex of a simplicial polytope.
    We use Gale's evenness condition. By definition the facets of 
    $\cM_{m,r}$ are given by the vertex set of a matching of size $r$ in
    $C_m$. If $M$ is such a set then for any two $i,j \in [m] \setminus M$
    such that $i < j$ the number $i < \ell < j$ of elements $\ell \in M$ that 
    lie between them must be even as they are exactly the elements covered by
    a set of disjoint edges. Thus by Gale's evenness condition it follows that 
    $M$ lies in $C(m,2r)$. Conversely, if $M$ is a facet of the boundary complex of
    $C(m,2r)$ then between any two $i, j \in [n] \setminus M$ where $i < j$ the number
    of $i < \ell < j$ is even. Thus by choosing $i$ and $j$ with the property that 
    $\{ \ell | j <  \ell < j \} \subseteq M$ one sees that by a partial matching of $C_m$ 
    one can cover all elements $\ell$ of $M$
    for which there are $i,j \in [m] \setminus M$ such that $i < \ell < j$. 
    Since in $M$ we have 
    $2r$ vertices of which an even number is covered, an even number is left.
    Those remaining vertices are an initial and final segment of $[m]$ and therefore
    can be covered by another few edges of $C_m$ that form a partial matching.
    Thus $M \in \cM_{m,r}$. 

    The rest of the claim now follows by standard facts about cyclic polytopes.
  \end{proof}
  
  The following lemma is certainly known, but we include it      
  for the sake of completeness.

  \begin{Lemma}
   \label{le:nonface}
   Let $r < \frac{m}{2}$.
   The minimal nonfaces of $\cM_{m,r}$ are the subsets $N$
   of $[m]$ such that
   \begin{itemize}
   \item[(i)] the cardinality of $N$ is $r+1$ and
   \item[(ii)] the set $N$ does not contain any edge of $C_m$.
   \end{itemize}
  \end{Lemma}

  \begin{proof}
   We show first show that each $N$ satisfying (i) and (ii)
   is a minimal nonface.
   Let $N \subseteq [m]$ be a set satisfying (i) and (ii).

   The set $N$ is a face of $\cM_{m,r}$
   if and only if we can find $r$ edges that cover $N$. But $N$ is of size
   $r+1$ and does not contain any edge. Hence $N$ cannot be covered by $r$
   edges and $N$
   is a nonface. Now let $v \in N$ be some vertex. Then $N\setminus \{v\}$
   contains $r$ elements. No two elements of $N \setminus \{v\}$ lie in an edge.
   Starting from any $w \in N \setminus \{v\}$ we go around $C_{2n}$
   in a fixed order.
   We pair each element of $N\setminus \{v\}$ with its neighbor in this
   order. Since no neighbor is in $N$ this will give $r$ edges covering
   $N \setminus \{v\}$. Hence $N\setminus \{v\}$ is a face.
   In particular $N$ is a minimal nonface.

   Now it remains to be shown that any minimal nonface $N$ of $\cM_{m,r}$
   satisfies (i) and (ii).
   Let $N$ be any minimal nonface of $\cM_{m,r}$. By $r < \frac{m}{2}$
   the full ground set $[m]$ is not a minimal nonface. Hence,
   there is a vertex $v$ that is
   not contained in $N$. Starting from $v$ we go in a fixed direction
   around $C_{m}$.
   We mark a vertex red if it is in $N$ and the preceding vertex is not
   yet marked red. We mark a vertex blue if it is in $N$ and the preceding
   vertex is marked red. We mark a vertex green if it is not in $N$ but
   the preceding vertex is marked red.
   It follows that $N$ consists of all red and blue
   vertices. Now remove from $N$ all blue vertices.
   Then the resulting set $N'$ does not contain any edge. Thus if
   $N'$ has $r+1$ or more elements then it contains a subset
   satisfying (i) and (ii). Since we know that all subsets satisfying (i) and
   (ii) are minimal nonfaces, it follows that $N$ itself must satisfy (i)
   and (ii).
   Hence we are left with the situation when $N'$ contains strictly
   less than $r+1$
   vertices. But by construction the vertex following a red vertex is
   either blue or green. Hence the set of  red, blue and green
   vertices is a set containing $N$ and being contained in a
   matching of size $r$. Thus $N$ cannot be a nonface.
  \end{proof}

 For exhibiting the connection of the previous lemmas with $\iin_\prec (I_{n-2})$, we consider a graph on vertex 
  set $D\cup U$, with 
  $$D = \{x_{11}, \ldots, x_{nn}\},\  \  U=\{ x_{13},x_{24}, \ldots, x_{n-2,n},x_{12},x_{n-1,n}\}.$$
  The edges are formed by the two elements subsets that contain one 
  element $x_{ii}$ in $D$ and one element in $U$ that lies either in the same
  row or column as $x_{ii}$ . One easily sees that this graph is a $2n$-cycle
  whose vertices alternate between elements in $D$ and elements in $U$ (see Figure 1).
    The preceding lemmas for $m = 2n$ and $r = n-3$ imply the
  following proposition.

  \begin{Proposition}
    \label{pr:initial2}
    The ideal $\iin_{\prec}(I_{n-2})$ is the Stanley-Reisner ideal of a simplicial complex
    isomorphic to an iterated cone over $\cM_{2n,n-3}$, resp. the boundary complex of $C(2n,2n-6)$.
  \end{Proposition}

  For the proof of the proposition we need a simple lemma that can for example
  be found in \cite[Lemma 5.1]{JonssonWelker2007}.  A version of this  for  arbitrary (not necessarily monomial)
  ideals is stated in \cite[Lemma 4.2]{ConcaHostenThomas2006}.

  \begin{Lemma} \label{monomialinclusion}
   Let $T = k[y_1, \ldots, y_\ell]$ be the polynomial ring in $\ell$
   variables. Suppose that $I \subseteq J$ are monomial  ideals in $T$
   such that the following hold:
   \begin{itemize}
    \item[(i)] $\dim(T/I) = \dim(T/J)$.
    \item[(ii)] $e(T/I) = e(T/J)$.
    \item[(iii)] $I$ is the Stanley-Reisner ideal of a pure simplicial complex $\Delta$
          on ground set $[\ell]$.
   \end{itemize}

   Then $I = J$.
  \end{Lemma}

  \begin{proof}[Proof of \ref{pr:initial2}]
   Let us identify the variables in $D \cup U$ with the elements of
   $[2n]$ as indicated in \ref{fig:figure1}.
   Then \ref{pr:initmon} implies that the monomials whose support sets
   are satisfy the conditions of \ref{le:nonface} lie in
   $\iin_\prec (I_{n-2})$. Thus \ref{le:nonface} implies that
   the Stanley-Reisner ideal of $\cM_{2n,n-3}$ is a subset of
   $\iin_\prec (I_{n-2})$.
   The dimensions of the respective quotient rings and their
   multiplicities coincide by \ref{le:cyclic} and \ref{co:n-2}.
   Hence using \ref{monomialinclusion}
   it follows that $\iin_\prec (I_{n-2})$ equals the
   Stanley-Reisner ideal of $\cM_{2n,n-3} \cong C(2n,2n-6)$.
  \end{proof}

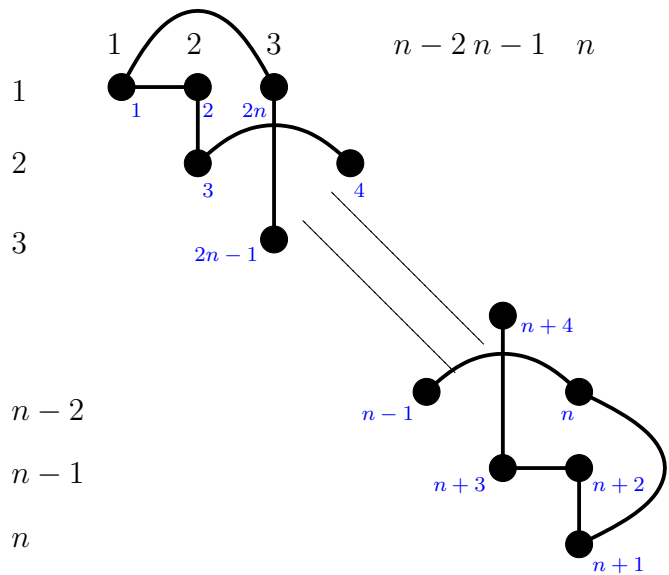
\begin{figure}
 \begin{picture}(0,0)%
  \hskip2.5cm\includegraphics[width=0.6\textwidth]{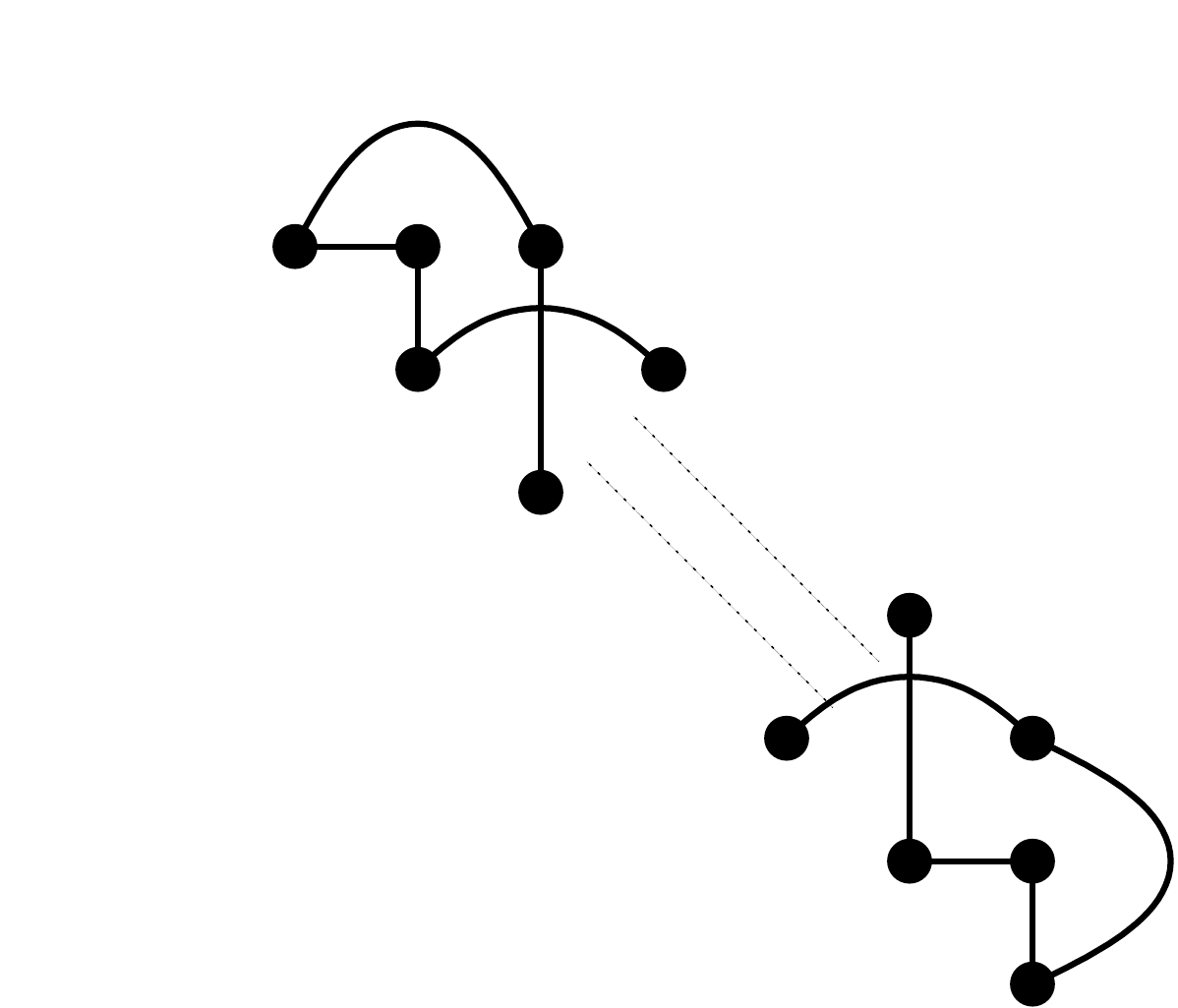}%
 \end{picture}%
 \setlength{\unitlength}{3947sp}%
 \begingroup\makeatletter\ifx\SetFigFont\undefined%
 \gdef\SetFigFont#1#2#3#4#5{%
  \reset@font\fontsize{#1}{#2pt}%
  \fontfamily{#3}\fontseries{#4}\fontshape{#5}%
  \selectfont}%
  \fi\endgroup%
 \begin{picture}(5748,4913)(961,-4074)
  \put(2600,-4000){\makebox(0,0)[lb]{\smash{{\SetFigFont{12}{14.4}{\rmdefault}{\mddefault}{\updefault}{\color[rgb]{0,0,0}$n$}}}}}
  \put(2600,-3600){\makebox(0,0)[lb]{\smash{{\SetFigFont{12}{14.4}{\rmdefault}{\mddefault}{\updefault}{\color[rgb]{0,0,0}$n-1$}}}}}
  \put(2600,-3200){\makebox(0,0)[lb]{\smash{{\SetFigFont{12}{14.4}{\rmdefault}{\mddefault}{\updefault}{\color[rgb]{0,0,0}$n-2$}}}}}
  \put(2600,-2150){\makebox(0,0)[lb]{\smash{{\SetFigFont{12}{14.4}{\rmdefault}{\mddefault}{\updefault}{\color[rgb]{0,0,0}$3$}}}}}
  \put(2600,-1650){\makebox(0,0)[lb]{\smash{{\SetFigFont{12}{14.4}{\rmdefault}{\mddefault}{\updefault}{\color[rgb]{0,0,0}$2$}}}}}
  \put(2600,-1200){\makebox(0,0)[lb]{\smash{{\SetFigFont{12}{14.4}{\rmdefault}{\mddefault}{\updefault}{\color[rgb]{0,0,0}$1$}}}}}

  \put(3200,-900){\makebox(0,0)[lb]{\smash{{\SetFigFont{12}{14.4}{\rmdefault}{\mddefault}{\updefault}{\color[rgb]{0,0,0}$1$}}}}}
  \put(3700,-900){\makebox(0,0)[lb]{\smash{{\SetFigFont{12}{14.4}{\rmdefault}{\mddefault}{\updefault}{\color[rgb]{0,0,0}$2$}}}}}
  \put(4200,-900){\makebox(0,0)[lb]{\smash{{\SetFigFont{12}{14.4}{\rmdefault}{\mddefault}{\updefault}{\color[rgb]{0,0,0}$3$}}}}}
  \put(5000,-900){\makebox(0,0)[lb]{\smash{{\SetFigFont{12}{14.4}{\rmdefault}{\mddefault}{\updefault}{\color[rgb]{0,0,0}$n-2$}}}}}
  \put(5500,-900){\makebox(0,0)[lb]{\smash{{\SetFigFont{12}{14.4}{\rmdefault}{\mddefault}{\updefault}{\color[rgb]{0,0,0}$n-1$}}}}}
  \put(6150,-900){\makebox(0,0)[lb]{\smash{{\SetFigFont{12}{14.4}{\rmdefault}{\mddefault}{\updefault}{\color[rgb]{0,0,0}$n$}}}}}

  \put(3350,-1300){\makebox(0,0)[lb]{\smash{{\SetFigFont{8}{10.0}{\rmdefault}{\mddefault}{\updefault}{\color{blue}$1$}}}}}
  \put(3800,-1300){\makebox(0,0)[lb]{\smash{{\SetFigFont{8}{10.0}{\rmdefault}{\mddefault}{\updefault}{\color{blue}$2$}}}}}
  \put(4050,-1300){\makebox(0,0)[lb]{\smash{{\SetFigFont{8}{10.0}{\rmdefault}{\mddefault}{\updefault}{\color{blue}$2n$}}}}}
  \put(3800,-1800){\makebox(0,0)[lb]{\smash{{\SetFigFont{8}{10.0}{\rmdefault}{\mddefault}{\updefault}{\color{blue}$3$}}}}}
  \put(4750,-1800){\makebox(0,0)[lb]{\smash{{\SetFigFont{8}{10.0}{\rmdefault}{\mddefault}{\updefault}{\color{blue}$4$}}}}}

  \put(3750,-2200){\makebox(0,0)[lb]{\smash{{\SetFigFont{8}{10.0}{\rmdefault}{\mddefault}{\updefault}{\color{blue}$2n-1$}}}}}
  \put(6250,-4150){\makebox(0,0)[lb]{\smash{{\SetFigFont{8}{10.0}{\rmdefault}{\mddefault}{\updefault}{\color{blue}$n+1$}}}}}
  \put(6250,-3650){\makebox(0,0)[lb]{\smash{{\SetFigFont{8}{10.0}{\rmdefault}{\mddefault}{\updefault}{\color{blue}$n+2$}}}}}
  \put(5250,-3650){\makebox(0,0)[lb]{\smash{{\SetFigFont{8}{10.0}{\rmdefault}{\mddefault}{\updefault}{\color{blue}$n+3$}}}}}
  \put(6050,-3200){\makebox(0,0)[lb]{\smash{{\SetFigFont{8}{10.0}{\rmdefault}{\mddefault}{\updefault}{\color{blue}$n$}}}}}
  \put(4800,-3200){\makebox(0,0)[lb]{\smash{{\SetFigFont{8}{10.0}{\rmdefault}{\mddefault}{\updefault}{\color{blue}$n-1$}}}}}
  \put(5800,-2650){\makebox(0,0)[lb]{\smash{{\SetFigFont{8}{10.0}{\rmdefault}{\mddefault}{\updefault}{\color{blue}$n+4$}}}}}
\end{picture}
\caption{Labeling of the cycle for $n$ even}
\label{fig:figure1}
\end{figure}

 The following is an immediate consequence of the well known fact
 that the Stanley-Reisner ideal of the boundary complex of a simplicial polytope
  is Gorenstein (see for example \cite[Corollary 5.5.6]{BrunsHerzog}) .
   
 \begin{Corollary}
  The simplicial complex $\cM_{m,r}$ is Gorenstein$^\ast$ for $r < \frac{m}{2}$.
  In particular, $\iin_\prec (I_{n-2})$ defines  a Gorenstein ring.
 \end{Corollary}

 Now \ref{main} follows noting that also $S/\iin_\prec (I_{n-2})$ is a 
 compressed Gorenstein $k$-algebras (see \ref{compressed}) with the same numerical 
 invariants as $S/I_{n-2}$, and arguing as before \ref{betti}.

\end{document}